\documentclass[11pt,reqno]{amsart}

\usepackage{mathrsfs}
\usepackage{amsfonts}
\usepackage{amssymb}
\usepackage{amsthm}
\usepackage{amsmath}
\usepackage{latexsym}
\usepackage{tikz}
\usepackage{geometry}

\newtheorem{theorem}{Theorem}[section]
\newtheorem{proposition}[theorem]{Proposition}
\newtheorem{lemma}[theorem]{Lemma}

\newtheorem{remark}[theorem]{Remark}

\makeatletter \@addtoreset{equation}{section} \makeatother

\def\tilde{\widetilde}

\newcommand{\beq}{\begin{equation}}
\newcommand{\eeq}{\end{equation}}

\makeatletter

\newcommand{\Rmnum}[1]{\expandafter\@slowromancap\romannumeral #1@}
\makeatother

\begin{document}

\title[2D  MHD System with only Magnetic Diffusion on Periodic Domain]
{Global Classical Solutions of 2D MHD System with only Magnetic Diffusion on Periodic Domain}

\author[Y. Zhou]{Yi Zhou}\address{School of Mathematical Sciences, Fudan University, Shanghai 200433, P.R.China}
\email{\tt yizhou@fudan.edu.cn}
\author[Y. Zhu]{Yi Zhu}\address{School of Mathematical Sciences, Fudan University, Shanghai 200433, P.R.China}
\email{\tt y\_zhu12@fudan.edu.cn}

\date{}
\let\thefootnote\relax\footnotetext{{\it{Submitted}}: Dec. 20, 2016.}
\subjclass[2010]{35Q35, 76D03, 76W05}
\keywords{ MHD System, Inviscid Flow, Global Classical Solutions.}

\begin{abstract}
This paper studies the global existence of classical solutions to the two-dimensional incompressible magneto-hydrodynamical (MHD) system with only magnetic diffusion on the periodic domain. The approach is based on a time-weighted energy estimate, under the assumptions that the initial magnetic field is close enough to an equilibrium state and the initial data have reflection symmetry.
\end{abstract}

\maketitle

\section{Introduction}
The dynamics of electrically conducting fluids interacting with magnetic fields can be described by the equations of magnetohydrodynamics (MHD).
When the viscosity of fluid is missing, it is extremely interesting and challenging to understand whether the magnetic diffusion only could prevent singularity development from small smooth initial data even in two dimensional physical space, in view of the strongly nonlinear coupling between fluids and magnetic field.
To this purpose, in this paper we will investigate the global existence of classical solutions to the following two dimensional incompressible magnetic hydrodynamical system
\begin{equation}\label{eq1.1}
\begin{cases}
B_t + u \cdot \nabla B -\Delta B= B\cdot \nabla u,\\
u_t + u \cdot \nabla u  + \nabla p= B\cdot \nabla B,\\
\nabla \cdot u = \nabla \cdot B = 0,\\
u(0,x)= u_0(x), \quad B(0,x) = B_0(x),
\end{cases}
\end{equation}
on the periodic domain
\begin{equation}\label{eq0.2}
x\in [-\pi, \pi]^2 = \mathbb{T}^2.
\end{equation}
Here $B=(B_1,B_2)^{\top}$ denotes the magnetic field, $u=(u_1,u_2)^{\top}$ the velocity and $p=q+\frac{1}{2}|B|^2$ where $q$ presents the scalar pressure of the fluid. This MHD system with only magnetic diffusion models many significant phenomenon such as the magnetic reconnection in astrophysics and geomagnetic dynamo in geophysics (see, e.g., \cite{EPTF}). For more physical background, we refer to \cite{HC, TGCDP, LDEM}.

  There is a growing literature devoted to these MHD systems so that it is almost impossible to be exhaustive in this introduction.
  To review the developments in this field systematically and exhaustively, we shall next classify the MHD systems into three kinds (accordingly to the level of dissipations) and recall the recent progress respectively. \\
$\mathbf{1. \; Viscous \; and \; resistive }$

For viscous and resistive MHD system (namely fully dissipative in fluids and in magnetic field), the existence and uniqueness results for both weak and strong solutions could be found in Duvaut and Lions \cite{DL} (we also refer to Sermange and Temam \cite{ST}).  Abidi and Paicu \cite{AP} generalized these results to the inhomogeneous system, their initial data is in the so called critical space. For the system with mixed partial viscosity and magnetic diffusion, Cao and Wu \cite{CW} (also see \cite{CRW}) derived the global well-posedness provided that initial data in $H^2(\mathbb{R}^2)$. \\
$\mathbf{2. \; Inviscid \; and \; non-resistive }$

For inviscid and non-resistive case (no viscosity, no magnetic diffusion, hence no dissipation), Bardos, Sulem and Sulem \cite{CBCS} first proved the global well-posedness of MHD system provided that the initial data is near a nontrivial equilibrium. This striking work made full use of the purely dispersion together with some coupling of nonlinearity between fluids and magnetic fields. Very recently, the vanishing dissipation limit from fully dissipative MHD system to this kind system has been justified by \cite{HXY, CL, WZ} under some structural conditions between viscosity and magnetic diffusion coefficients. \\
$\mathbf{3. \; Partially \; dissipative }$

For the remaining case, partially dissipative MHD system (only viscosity or magnetic diffusion presents), we first focus on the system with positive viscosity and zero resistivity. Many excellent works have just been made in recent years about this system. Lin and Zhang \cite{FLPZ} established the global well posedness for a closely related model in 3D (also see \cite{FLTZ}). With certain admissible condition for initial data, Lin, Xu and Zhang \cite{FLLX} proved the global existence of classical solutions in 2D when the initial data is close to a nontrivial equilibrium state. The same problem in 3D was obtained by Xu and Zhang \cite{LXPZ}. Ren, Wu, Xiang and Zhang \cite{REN} first removed the restriction in 2D case (more results in 2D we refer to \cite{TZ, HU, HU2}).
 Very recently, Abidi and Zhang \cite{HAPZ} get the global well posedness for 3D case without the admissible restriction. For the initial boundary value problem, we refer to \cite{RXZ2, TW, PZZ}.
For somehow large data, Lei \cite{Lei2} proved the global regularity of some axially symmetric solutions in three dimensional space to this kind MHD system.

In the case of our consideration, namely the MHD system with zero viscosity and positive resistivity, the global well posedness is still a challenging open problem even in two dimensional space. Due to the absence of viscosity, the behaviour of fluid seems to be extremely wild which brings the main difficulty. To the best of our knowledge, the system only admits a $H^1$ global weak solution \cite{LZ} (also see \cite{CW}). For the research of this kind generalized MHD system and regularity criterion, we refer to \cite{CWY, JZ, JNWXY}. To go forward, it is then natural to explore whether the MHD system with only magnetic diffusion can develop global classical solutions.

 The aim of our work is to establish the global existence of smooth solutions to the 2D incompressible magneto-hydrodynamical system with periodic boundary conditions, under the assumptions that the initial magnetic field is close enough to an equilibrium state and the initial data have some particular symmetries.

Our assumptions for the initial velocity and magnetic field are as follows  \\
\begin{equation}\label{eq0.3}
\begin{split}
u_{0,1}(x), \quad B_{0,2}(x)\quad &\text{are even periodic with respect to } x_2,\\
u_{0,2}(x),\quad  B_{0,1}(x) \quad &\text{are odd periodic with respect to } x_2.
\end{split}
\end{equation}
Moreover, we assume
\begin{equation}\label{eq0.4}
\int_{\mathbb{T}^2} u_0 \;dx = 0, \quad \int_{\mathbb{T}^2} B_{0,2} \;dx = \alpha \neq 0.
\end{equation}
Under such assumptions, now we can state our main result in the following theorem.
\begin{theorem}\label{thm1}
Consider the 2D incompressible MHD system \eqref{eq1.1}-\eqref{eq0.2} with initial data satisfies the conditions \eqref{eq0.3}-\eqref{eq0.4}. Assume that $(u_0, B_0) \in H^{2s + 1}(\mathbb{T}^2)$ with integer $s \geq 2$ and $\nabla \cdot u_0 = \nabla \cdot B_0 = 0$. Then there exists a small constant $\epsilon > 0$ depending  on $\alpha$ such that the system \eqref{eq1.1} admits a
global classical solution provided that
\begin{equation}\label{setofdata}
\|u_0\|_{H^{2s+1}} + \|\nabla B_0\|_{H^{2s}}\leq \epsilon.
\end{equation}
\end{theorem}

Without loss of generality, we assume $\alpha = (2\pi)^2$. Following Lin and Zhang \cite{FLPZ}, we let $B_0=b_0 + e_2$, where $e_2=(0,1)^{\top}$. Hence, we can count
\begin{equation}\label{eq1.02}
\int_{\mathbb{T}^2} b_0 \;dx = 0.
\end{equation}
Let $B=b+e_2$, we can derive the system of  $(u, b)$ from \eqref{eq1.1}
\begin{equation}\label{eq1.2}
\begin{cases}
b_t + u \cdot \nabla b - \Delta b= b\cdot \nabla u + \partial_2 u,\\
u_t + u \cdot \nabla u  + \nabla p= b\cdot \nabla b + \partial_2 b,\\
\nabla \cdot u = \nabla \cdot b = 0,\\
u(0,x)= u_0(x), \quad b(0,x) = b_0(x).
\end{cases}
\end{equation}
And the properties of initial data \eqref{eq0.3}-\eqref{eq0.4} persist. Indeed,
\begin{equation}\label{eq1.3}
\begin{split}
u_1(t,x), \quad b_2(t,x)\quad &\text{are even periodic with respect to } x_2,\\
u_2(t,x),\quad  b_1(t,x) \quad &\text{are odd periodic with respect to } x_2.
\end{split}
\end{equation}
\begin{equation}\label{eq1.4}
\int_{\mathbb{T}^2} b \;dx = \int_{\mathbb{T}^2} u \;dx = 0.
\end{equation}

\begin{remark}\label{rem}
   Define $ \tilde{u}(t,x), \tilde{b}(t,x) $ as follows
  \begin{equation}\nonumber
\begin{split}
\tilde{u}_1(t, x_1, x_2) =& u_1(t,x_1, - x_2), \quad \tilde{u}_2(t,x_1, x_2) = - u(t,x_1,- x_2), \\
\tilde{b}_1(t, x_1, x_2) =& - b_1(t,x_1, -x_2), \quad \tilde{b}_2(t,x_1, x_2) = b_2 (t, x_1, -x_2).
\end{split}
\end{equation}
  It easily shows that $\tilde{u}, \tilde{b}$ satisfy the same system \eqref{eq1.2} like $u, b$ with the same initial data. Hence, by uniqueness of solution, we know that $\tilde{b}(t,x) = b(t,x)$ and $\tilde{u}(t,x) = u(t,x)$. Thus, the property \eqref{eq1.3} holds.
\end{remark}

  After the reformulation, we can turn to the system \eqref{eq1.2} instead of  considering the system \eqref{eq1.1} directly.

  In the proof of our theorem, we have to face the difficulties arising from no viscosity and bounded domain.
  These challenges will be overcome through some carefully designed weighted energies with the help of some observations to the structure of the system.
  One of our major observations is that Poincar\'{e} inequality can compensate the loss of possible dispersion generated by bounded domain. Another observation is that thanks to the
   symmetric condition \eqref{eq1.3},  we can capture the dissipation of velocity more easily. Now let us explain it more thoroughly. We all know that the motion of fluid is motivated by the velocity $u$. Thus, to preserve the
   regularity, one have to control the key quantity $\nabla u$ which seems to be quite difficult due to the lack of viscosity.  With the help of divergence free condition, we now study the vorticity $\omega = \partial_1 u_2 - \partial_2 u_1 $ instead. Under the symmetric condition \eqref{eq1.3}, it easily knows that $\omega(x)$ is odd periodic with respect to $x_2$ which implies the zero  integral average in $x_2$ direction immediately. Very naturally, the Poincar\'{e} inequality will be applied and the norms of $\omega$ can be controlled by the norms of $\partial_2\omega$. For more details, we can see Proposition \ref{prop} in Section 2.

\begin{remark}\label{rem2}
Finally, we point out the difficulty in the generalization to the three dimensional case. By choosing the initial magnetic field $B_0=e_3$, the 3D MHD system with zero viscosity and positive resistive reduces to the 2D incompressible Euler equations. As known in \cite{ZA, LEI}, on the torus, the gradient of vorticity $\|\nabla \omega(t)\|_{L^{\infty}}$ will grow exponentially in time with some initial data. Hence, we may not expect the same result as we obtained in this paper in the 3D case.
\end{remark}

\section{Preliminary}

In this section, we make some preparations for the proof of Theorem \ref{thm1} in Section 3.

First, we introduce a useful proposition based on Poincar\'{e} inequality which gratefully helps to capture the dissipation of velocity.
\begin{proposition}\label{prop}
For any vector $u(x) \in H^{k+2}(\mathbb{T}^2)$, $k \in \mathbb{N}$ satisfying $\nabla \cdot u=0$ and the following condition
\begin{equation}\label{eqp1}
\begin{split}
u_1(x) &\quad \text{is even periodic with respect to } x_2,\\
u_2(x)&\quad \text{is odd periodic with respect to } x_2.
\end{split}
\end{equation}
It holds that
\begin{equation}\nonumber
 \|\nabla u\|_{H^k (\mathbb{T}^2)} \lesssim  \|\partial_2 u\|_{H^{k+1} (\mathbb{T}^2)} .
\end{equation}
\end{proposition}

\begin{proof}

We define the quantity $\omega = \partial_1u_2-\partial_2u_1$ in the two dimensional space and then we have the following   equality
$$ \Delta u=\nabla\nabla \cdot u+\nabla^{\perp}\omega.$$
Here $\nabla^{\perp}=(-\partial_2,\partial_1)$.

On one hand, notice the divergence free condition, we have
$$ \nabla u= -\nabla (-\Delta)^{-1}\nabla^{\perp}\omega .$$
Using standard Calderon-Zygmund theory $\|Zg\|_{L^p} \lesssim \|g\|_{L^p}$ for Riesz operator $Z$ and $1<p<\infty$, we can bound
\begin{equation}\label{eqpp}
\|\nabla u\|_{H^{k}} \lesssim \|\omega\|_{H^{k}}.
\end{equation}

On the other hand, due to the definition of $\omega$ and the assumption \eqref{eqp1}, we easily know that $\omega$ is odd periodic with respect to $x_2$ which implies
\begin{equation}\label{eqppp}
\frac{1}{2\pi}\int_{-\pi}^{\pi} \omega(x_1,x_2) dx_2 =0, \quad \forall x_1 \in [-\pi,\pi].
\end{equation}
For multi-index $\alpha = (\alpha_1, \alpha_2)$, indeed $\partial ^{\alpha} = \partial_1^{\alpha_1}\partial_2^{\alpha_2}$, we can write
\begin{equation}\label{eqp2}
\|\omega\|^2_{H^k (\mathbb{T}^2)} =\sum_{|\alpha|=0}^k \int_{-\pi}^{\pi} \int_{-\pi}^{\pi} |\partial^{\alpha}\omega(x_1,x_2)|^2 dx_2 dx_1  .
\end{equation}
Notice the fact \eqref{eqppp}, we have, for multi-index $\alpha=(\alpha_1,0)$
\begin{equation}\nonumber
  \frac{1}{2\pi}\int_{-\pi}^{\pi} \partial^{\alpha}\omega(x_1,x_2) dx_2 =0.
\end{equation}
Also, for multi-index $\alpha=(\alpha_1,\alpha_2)$ where $\alpha_2 >0$, using the periodic boundary condition, we have
\begin{equation}\nonumber
  \int_{-\pi}^{\pi} \partial^{\alpha}\omega(x_1,x_2) dx_2 = \partial_1^{\alpha_1}\partial_2^{\alpha_2-1} \omega(x_1,\cdot)\Big|_{-\pi}^{\pi} =0.
\end{equation}
Therefore, the integral average of $\partial^{\alpha} \omega(x_1, x_2)$ in $x_2$ direction over $[-\pi, \pi]$ is zero.
Hence, applying standard Poincar\'{e} inequality to $\partial^{\alpha}\omega(x_1,x_2)$ in $x_2$ direction, we have $\forall x_1 \in [-\pi,\pi]$:
\begin{equation}\nonumber
  \int_{-\pi}^{\pi} |\partial^{\alpha}\omega(x_1,x_2)|^2 dx_2 \lesssim \int_{-\pi}^{\pi} |\partial^{\alpha}\partial_2\omega(x_1,x_2)|^2 dx_2.
\end{equation}
According to the definition of $\|\omega\|_{H^k(\mathbb{T}^2)}$ i.e. \eqref{eqp2}, we finally obtain
\begin{equation}\label{eqpp1}
\begin{split}
\|\omega\|^2_{H^k (\mathbb{T}^2)} =&\sum_{|\alpha|=0}^k \int_{-\pi}^{\pi} \int_{-\pi}^{\pi} |\partial^{\alpha}\omega(x_1,x_2)|^2 dx_2 dx_1 \\
\lesssim & \sum_{|\alpha|=0}^k \int_{-\pi}^{\pi} \int_{-\pi}^{\pi} |\partial^{\alpha}\partial_2\omega(x_1,x_2)|^2 dx_2 dx_1 \\
= &\|\partial_2\omega\|^2_{H^k (\mathbb{T}^2)}.
\end{split}
\end{equation}

Finally, combining \eqref{eqpp} and \eqref{eqpp1} together, we finish the proof of this proposition
\begin{equation}\nonumber
\begin{split}
\|\nabla u\|_{H^{k}(\mathbb{T}^2)} \lesssim \|\omega\|_{H^{k}(\mathbb{T}^2)}
\lesssim \|\partial_2 \omega\|_{H^{k}(\mathbb{T}^2)}
 \lesssim \|\partial_2 u\|_{H^{k+1}(\mathbb{T}^2)}.
\end{split}
\end{equation}

\end{proof}

Now, we introduce the energies which will enable us to achieve our desired estimates. Based on the discussion in Section 1, we define some time-weighted energies for the system \eqref{eq1.2}. The energies below are defined on the domain $\mathbb{R}^+ \times\mathbb{T}^2$. For $s\in \mathbb{N}$, we give
\begin{equation}\label{eqframe}
\begin{split}
\mathcal{E}_0(t) =& \sup_{0\leq \tau \leq t}  \big(\|u(\tau)\|_{H^{2s+1}}^2+\|b(\tau)\|_{H^{2s+1}}^2 \big)\\
&+ \int_{0}^{t}  \big(\|b(\tau)\|_{H^{2s+2}}^2+\|\partial_2u(\tau)\|_{H^{2s}}^2\big) \;d\tau,\\
\mathcal{E}_1(t)=& \sup_{0\leq \tau \leq t} (1+\tau)^2 \big(\|u(\tau)\|_{H^{2s-1}}^2+\|b(\tau)\|_{H^{2s-1}}^2\big)\\
&+ \int_{0}^{t}  (1+\tau)^2\big(\|b(\tau)\|_{H^{2s}}^2+\|\partial_2u(\tau)\|_{H^{2s-2}}^2\big) \; d\tau.
\end{split}
\end{equation}

In the next section, we will derive the uniform estimate of $\mathcal{E}_0(t)$ and $\mathcal{E}_1(t)$ and then give the proof of Theorem \ref{thm1}. When doing the estimate, we only need to consider the highest order norms in each energy due to \eqref{eq1.4} and Poincar$\mathrm{\acute{e}}$ inequality.

\section{Energy estimate and the proof of  main result}
\subsection{ \textit{A priori} estimate}
We shall first deal with the higher order energy $\mathcal{E}_0(t)$ and prove the Lemma 3.1. It shows the uniform bound of highest order norm $H^{2s+1}(\mathbb{T}^2)$ for both $u(t,\cdot)$ and $b(t,\cdot)$.
\begin{lemma}\label{lem1}
For $s\geq 2$, we have
\begin{equation}\label{E_0}
\mathcal{E}_0(t)\lesssim \mathcal{E}_0(0)+\mathcal{E}_0^{3/2}(t)+\mathcal{E}_0^{9/8}(t)\mathcal{E}_1^{3/8}(t).
\end{equation}
\end{lemma}

\begin{proof}
We divide the proof into two parts. Instead of deriving the estimate of $\mathcal{E}_0(t)$ directly, we will first get the uniform bound of $\mathcal{E}_{0,1}(t)$ which defined as follows
\begin{equation}\label{eqE01}
\mathcal{E}_{0,1}(t) \triangleq    \sup_{0\leq \tau \leq t}  \big(\|u(\tau)\|_{H^{2s+1}}^2+\|b(\tau)\|_{H^{2s+1}}^2 \big)
+ \int_{0}^{t}   \|b(\tau)\|_{H^{2s+2}}^2 \; d\tau .
\end{equation}

To get the estimate of $\mathcal{E}_{0,1}$, we first operate $\nabla^{2s+1}$ derivative on the system \eqref{eq1.2}. Then, taking inner product with $\nabla^{2s+1} b$ for the first equation of system \eqref{eq1.2} and taking inner product with $\nabla^{2s+1} u$ for the second equation of system \eqref{eq1.2}, we get

 \begin{equation}\label{lem2.1}
 \frac{1}{2} \frac{d}{dt} \big( \|b\|_{\dot H^{2s+1}}^2+\|u\|_{\dot H^{2s+1}}^2 \big)+\|b\|_{\dot H^{2s+2}}^2=M_1+M_2+M_3+M_4,
 \end{equation}

where,
\begin{equation}\nonumber
\begin{split}
  M_1=&-\int_{\mathbb{T}^2}u\cdot \nabla \nabla^{2s+1} u \; \nabla^{2s+1} u + u\cdot \nabla \nabla^{2s+1} b \; \nabla^{2s+1} b \;dx\\
  &+\int_{\mathbb{T}^2}b\cdot \nabla \nabla^{2s+1} u \; \nabla^{2s+1} b + b\cdot \nabla \nabla^{2s+1} b \; \nabla^{2s+1} u \;dx\\
  &+\int_{\mathbb{T}^2}\nabla^{2s+1} \partial_2u \nabla^{2s+1}b+\nabla^{2s+1} \partial_2b \nabla^{2s+1}u\; dx ,\\
  M_2=& \sum_{k=1}^{s}C_{2s+1}^{k}\int_{\mathbb{T}^2} (\nabla^{k}b \cdot \nabla \nabla^{2s+1-k} u- \nabla^{k} u \cdot \nabla
\nabla^{2s+1-k} b) \nabla^{2s+1} b\;dx\\
&+\sum_{k=s+1}^{2s+1}C_{2s+1}^{k}\int_{\mathbb{T}^2} (\nabla^{k}b \cdot \nabla \nabla^{2s+1-k} u- \nabla^{k} u \cdot \nabla
\nabla^{2s+1-k} b) \nabla^{2s+1} b\;dx ,\\
M_3=&- \sum_{k=1}^{2s+1}C_{2s+1}^{k}\int_{\mathbb{T}^2} \nabla^{k}u\cdot \nabla \nabla^{2s+1-k} u\; \nabla^{2s+1}u \;dx,\\
M_4=&\sum_{k=1}^{2s+1}C_{2s+1}^{k}\int_{\mathbb{T}^3}\nabla^{k}b \cdot \nabla \nabla^{2s+1-k} b\; \nabla^{2s+1} u\;dx .
\end{split}
\end{equation}

First, for the term $M_1$, using integration by parts and divergence free condition, we have
\begin{equation}\label{eqI1}
M_1 = 0.
\end{equation}

The main idea of the following estimates is to derive the bound of each term carefully so that it can be controlled by the combination of $\mathcal{E}_0(t)$ and $\mathcal{E}_1(t)$. By H\"{o}lder inequality and Sobolev imbedding theorem, for $s\geq 2$, we have
\begin{equation}\nonumber
\begin{split}
|M_2| \lesssim & \|u\|_{H^{2s+1}}\|b\|_{W^{s,\infty}}\|b\|_{H^{2s+1}}+\|u\|_{W^{s,\infty}}\|b\|_{H^{2s+1}}^2\\
&+\|b\|_{H^{2s+1}}\|u\|_{W^{s+1,\infty}}\|b\|_{H^{2s+1}}+\|b\|_{W^{s+1,\infty}}\|u\|_{H^{2s+1}}\|b\|_{H^{2s+1}}\\
 \lesssim & \|u\|_{H^{2s+1}}\|b\|_{H^{2s+2}}^2.
\end{split}
\end{equation}
Hence,
\begin{equation}\label{eqI2}
\begin{split}
\int_{0}^{t} |M_2(\tau)| \;d\tau \lesssim &\sup_{0 \leq \tau \leq t} \|u(\tau)\|_{H^{2s+1}}\int_{0}^{t} \|b(\tau)\|_{H^{2s+2}}^2\; d\tau\\
 \lesssim& \mathcal{E}^{3/2}_0(t).
\end{split}
\end{equation}

For the estimate of $M_3$, we should point out that it is the wildest term in our proof due to the bad behaviour of $u\cdot \nabla u$, indeed the inviscid property of our system.
Through the analysis of system \eqref{eq1.2}, we easily know $\partial_2u$ is a good term so that $u_2\cdot \nabla_2 u$ can be controlled well. While the estimate for $u_1\cdot\nabla_1 u$ is still nontrivial. Thanks to the Proposition \ref{prop} proved in Section 2, we can overcome this problem by introducing the vorticity $\omega$ as a connection. Our main strategy can be stated as follows.

Due to the divergence free condition of $u$ and the property \eqref{eq1.3}, we can bound the norm of $\nabla u$ by controlling $\omega$, thus $\partial_2 \omega$, indeed $\partial_2 \nabla u$.
Now, using H\"{o}lder inequality and Gagliardo--Nirenberg interpolation inequality,
\begin{equation}\nonumber
\begin{split}
  |M_3|\lesssim & \sum_{k = 1}^{2s+1} \|\nabla ^{k} u\|_{L^\frac{4s}{k-1}} \|\nabla ^{2s+2-k} u\|_{L^\frac{4s}{2s+1-k}} \|\nabla ^{2s+1} u\|_{L^2}\\
\lesssim & \sum_{k = 1}^{2s+1} \|\nabla u\|_{L^\infty}^\frac{2s+1-k}{2s} \|\nabla u\|_{H^{2s}}^\frac{k - 1}{2s} \|\nabla u\|_{L^\infty} ^\frac{k - 1}{2s}\|\nabla u\|_{H^{2s}}^\frac{2s+1-k}{2s}\|u\|_{H^{2s+1}}\\
 \lesssim & \|\nabla u\|_{L^\infty}\|u\|_{H^{2s+1}}^2\\
 \lesssim & \|\nabla u\|_{H^{2s-1}}^{1/4} \|\nabla u\|_{H^{2s-3}}^{3/4} \|u\|_{H^{2s+1}}^2 .
\end{split}
\end{equation}
Notice here $\frac{4s}{k-1}$ can be treated as $\infty$ when $k=1$, and similar holds for $\frac{4s}{2s+1-k}$. Also in the last step, we use condition $s \geq 2$ for the interpolation inequality.
Thus according to Proposition \ref{prop}, we directly have
\begin{equation}\nonumber
|M_3|
  \lesssim \|\partial_2 u\|_{H^{2s}}^{1/4}\| \partial_2 u\|_{H^{2s-2}}^{3/4}\|u\|_{H^{2s+1}}^2.
\end{equation}
Hence,
\begin{equation}\label{eqI3}
\begin{split}
&\int_{0}^{t} |M_3(\tau)| \;d\tau \\
\lesssim& \sup_{0 \leq \tau \leq t}\|u(\tau)\|_{H^{2s+1}}^2 \cdot \int_{0}^{t} \frac{1}{(1+\tau)^{3/4}}\|\partial_2 u\|_{H^{2s}}^{1/4}(1+\tau)^{3/4}\| \partial_2 u\|_{H^{2s-2}}^{3/4}
\;d\tau\\
\lesssim & \mathcal{E}^{9/8}_0(t)\mathcal{E}^{3/8}_1(t).
\end{split}
\end{equation}

For the last term $M_4$, we use the same method as above
\begin{equation}\nonumber
\begin{split}
  |M_4|\lesssim & \|b\|_{W^{s+1,\infty}}  \|b\|_{H^{2s+1}}   \|u\|_{H^{2s+1}}\\
  \lesssim & \| b\|_{H^{2s+2}}^2  \|u\|_{H^{2s+1}}.
\end{split}
\end{equation}
Thus we can bound,
\begin{equation}\label{eqI4}
\begin{split}
\int_{0}^{t} |M_4(\tau)| \;d\tau \lesssim& \sup_{0 \leq \tau \leq t} \|u(\tau)\|_{H^{2s+1}} \int_{0}^{t}
\|b(\tau)\|_{H^{2s+2}}^2\; d\tau \\
\lesssim &
 \mathcal{E}^{3/2}_0(t).
 \end{split}
\end{equation}

Summing up the estimates for $M_1 \thicksim M_4$, i.e., \eqref{eqI1}, \eqref{eqI2}, \eqref{eqI3} and \eqref{eqI4}. Integrating \eqref{lem2.1} with time, we now get the estimate of $\mathcal{E}_{0,1}(t)$ which is defined in \eqref{eqE01},
\begin{equation}\label{E01}
\mathcal{E}_{0,1}(t) \lesssim  \mathcal{E}_0(0)+\mathcal{E}^{3/2}_0(t)+\mathcal{E}^{9/8}_0(t)\mathcal{E}^{3/8}_1(t).
\end{equation}
Here, we have used the Poincar$\mathrm{\acute{e}}$ inequality to consider the highest order norm only.

Next, we work with the left term in $\mathcal{E}_0(t)$. Operating $\nabla^{2s}$ derivative on the first equation of system \eqref{eq1.2} and taking inner product with $\nabla^{2s} \partial_2 u$, we get

\begin{equation}\label{eqlem1.2}
\|\partial_2 u\|_{\dot H^{2s}}^2 = M_5 + M_6 + M_7,
\end{equation}
where
\begin{equation}\nonumber
\begin{split}
  M_5 =& \int_{\mathbb{T}^2} \nabla^{2s} (u\cdot \nabla b-b\cdot \nabla u) \nabla^{2s}\partial_2 u \; dx ,\\
  M_6 =& -\int_{\mathbb{T}^2}\nabla^{2s} \Delta b \nabla^{2s}\partial_2 u \; dx ,\\
  M_7 =& \int_{\mathbb{T}^2} \nabla^{2s} b_t \nabla^{2s}\partial_2 u \; dx .
\end{split}
\end{equation}

 As the process in the estimate of $\mathcal{E}_{0,1}$, we shall derive the estimate of each term on the right hand side of \eqref{eqlem1.2}. Using H\"{o}lder inequality and Sobolev imbedding theorem, for $s\geq 2$, we can bound
\begin{equation}\nonumber
\begin{split}
  |M_5| \lesssim &\|u\|_{W^{s+1,\infty}} \|b\|_{H^{2s+1}}\|\partial_2 u\|_{H^{2s}} +\|b\|_{W^{s+1,\infty}} \|u\|_{H^{2s+1}}\|\partial_2 u\|_{H^{2s}}   \\
  \lesssim &\|u\|_{H^{2s+1}}\|b\|_{H^{2s+1}}\|\partial_2 u\|_{H^{2s}}.
\end{split}
\end{equation}
Thus,
\begin{equation}\label{eqI5}
\begin{split}
&\int_{0}^{t} |M_5(\tau)| \; d\tau\\
 \lesssim & \sup_{0\leq \tau \leq t} \|u(\tau)\|_{H^{2s+1}} \int_{0}^{t} \|b\|_{H^{2s+1}}\|\partial_2 u\|_{H^{2s}}\; d\tau \\
\lesssim & \mathcal{E}^{3/2}_0(t) .
\end{split}
\end{equation}

The estimate for $M_6$ is almost the same, we directly say
\begin{equation}\label{eqI6}
  \begin{split}
\int_{0}^{t} |M_6(\tau)| \; d\tau \lesssim & \int_{0}^{t} \|b\|_{H^{2s+2}}\|\partial_2 u\|_{H^{2s}} \; d\tau\\
\lesssim & \mathcal{E}_{0,1}^{1/2}(t) \Big( \int_{0}^{t}\|\partial_2 u\|_{H^{2s}}^2 \; d\tau \Big)^{1/2} .
\end{split}
\end{equation}

For the last term $M_7$, using integration by parts, we can write this term into
\begin{equation}\label{eq2.19}
  M_7 = \frac{d}{dt}\int_{\mathbb{T}^2} \nabla^{2s}b \nabla^{2s}\partial_2 u \;dx +\int_{\mathbb{T}^2} \nabla^{2s}\partial_2 b \nabla^{2s}u_t \;dx .
\end{equation}
 Notice the second equality of system \eqref{eq1.2}, the second part on the right hand side of \eqref{eq2.19} can be estimated as follows
\begin{equation}\nonumber
\begin{split}
  & \int_{\mathbb{T}^2}  \nabla^{2s}\partial_2 b \nabla^{2s} \big( b\cdot \nabla b+\partial_2 b-u\cdot \nabla u-\nabla p\big) \; dx \\
  \lesssim & \|b\|_{W^{s,\infty}}\|b\|_{H^{2s+1}}^2 + \|b\|_{H^{2s+1}}^2 +\|b\|_{H^{2s+1}} \|\nabla u\|_{W^{s-1,\infty}}\|u\|_{H^{2s+1}} \\
  \lesssim & \|b\|_{H^{2s}}\|b\|_{H^{2s+1}}^2 +\|b\|_{H^{2s+1}}^2+\|b\|_{H^{2s+1}} \|\partial_2 u\|_{H^{2s}}\|u\|_{H^{2s+1}},
\end{split}
\end{equation}
provided that $ s\geq 2$. Thus,
\begin{equation}\label{eqI7}
\int_{0}^{t} |M_7(\tau)| \; d\tau \lesssim \mathcal{E}_{0,1}(t)+\mathcal{E}^{3/2}_0(t).
\end{equation}

Integrating \eqref{eqlem1.2} with time, using \eqref{eqI5}, \eqref{eqI6}, \eqref{eqI7}, Young inequality and Poincar$\mathrm{\acute{e}}$ inequality, we can bound

\begin{equation}\label{eqE02}
\int_{0}^{t} \|\partial_2 u\|_{H^{2s}}^2 \; d\tau \lesssim \mathcal{E}_{0,1}(t)+\mathcal{E}^{3/2}_0(t).
\end{equation}

Multiplying \eqref{E01} by suitable large number and plus \eqref{eqE02}, we then complete the proof of this lemma.

\end{proof}

Next, we want to give the estimate of lower order energy $\mathcal{E}_1(t)$ defined in \eqref{eqframe}. We give the following lemma.

\begin{lemma}\label{lem2}
For $s \geq 2$, we have
\begin{equation}\label{E_1}
\mathcal{E}_1(t)\lesssim \mathcal{E}_1(0)+\mathcal{E}_0^{1/2}(t)\mathcal{E}_1^{1/2(t)}+\mathcal{E}_0^{1/8}(t)\mathcal{E}_1^{11/8}(t)+\mathcal{E}_0^{1/2}(t)\mathcal{E}_1(t)+\mathcal{E}_1^{3/2}(t).
\end{equation}
\end{lemma}

\begin{proof}
Like the proof above, we divide the proof into two parts. Also, we will first get the estimate of $\mathcal{E}_{1,1}$ which defined in the following
\begin{equation}\nonumber
\mathcal{E}_{1,1}(t) \triangleq   \sup_{0\leq \tau \leq t}  (1+\tau)^2 \big(\|u(\tau)\|_{H^{2s-1}}^2+\|b(\tau)\|_{H^{2s-1}}^2 \big)
+ \int_{0}^{t}   (1+\tau)^2\|b(\tau)\|_{H^{2s}}^2 \; d\tau .
\end{equation}

Operating $\nabla^{2s-1}$ derivative on  system \eqref{eq1.2}, taking inner product with $\nabla^{2s-1} b$ for the first equation of \eqref{eq1.2} and taking inner product with $\nabla^{2s-1} u$ for the second equation of \eqref{eq1.2}. Adding the time weight $(1+t)^2$, we get
 \begin{equation}\label{lem2.2}
 \frac{1}{2} \frac{d}{dt} (1+t)^2 \big( \|b\|_{\dot H^{2s-1}}^2+\|u\|_{\dot H^{2s-1}}^2 \big)+ (1+t)^2\|b\|_{\dot H^{2s}}^2=N_1+N_2+N_3+N_4,
 \end{equation}
where,
\begin{equation}\nonumber
\begin{split}
N_1 =& (1+t)\big( \|b\|_{\dot H^{2s-1}}^2+\|u\|_{\dot H^{2s-1}}^2 \big), \\
N_2= & (1+t)^2 \int_{\mathbb{T}^2} \nabla^{2s-1}(b\cdot \nabla u-u\cdot\nabla b) \nabla^{2s-1} b \; dx ,\\
N_3 =& (1+t)^2 \int_{\mathbb{T}^2}\nabla^{2s-1}(b\cdot\nabla b)\nabla^{2s-1} u \; dx ,\\
N_4 =& -(1+t)^2 \int_{\mathbb{T}^2}\nabla^{2s-1}(u\cdot\nabla u)\nabla^{2s-1} u \; dx .
\end{split}
\end{equation}

Similarly, for the term $N_1$, using Proposition \ref{prop} and Gagliardo--Nirenberg interpolation inequality, we have
\begin{equation}\nonumber
\begin{split}
  |N_1|
  \lesssim &(1+t)\big( \|\partial_2 u\|_{ H^{2s-1}}^2+\|b\|_{ H^{2s-1}}^2 \big) \\
  \lesssim & \|\partial_2u\|_{ H^{2s}} (1+t)\|\partial_2u\|_{ H^{2s-2}}+\|b\|_{ H^{2s+2}} (1+t)\|b\|_{ H^{2s}} .
\end{split}
\end{equation}
Thus,
\begin{equation}\label{eqJ1}
\begin{split}
&\int_{0}^{t} |N_1(\tau)| \; d\tau \\
\lesssim &
\big(\int_{0}^{t} \|\partial_2u\|_{ H^{2s}}^2 \; d\tau \big)^{1/2}\cdot \big(\int_{0}^{t} (1+\tau)^2\|\partial_2u\|_{ H^{2s-2}}^2 \; d\tau \big)^{1/2}\\
&+ \big(\int_{0}^{t} \|b\|_{ H^{2s+2}}^2 \; d\tau \big)^{1/2}\cdot \big(\int_{0}^{t} (1+\tau)^2\|b\|_{ H^{2s}}^2 \; d\tau \big)^{1/2}\\
\lesssim &
\mathcal{E}_0^{1/2}(t)\mathcal{E}_1^{1/2}(t).
\end{split}
\end{equation}

For the term $N_2$, using H\"{o}lder inequality and Sobolev imbedding theorem,  notice $s \geq 2$ we have
\begin{equation}\nonumber
\begin{split}
  |N_2| \lesssim & (1+t)^2\|b\|_{W^{s,\infty}} \| u\|_{H^{2s}}\|b\|_{H^{2s-1}} +(1+t)^2\| u\|_{W^{s,\infty}} \|b\|_{H^{2s}}\|b\|_{H^{2s-1}}   \\
  \lesssim &(1+t)^2\| u\|_{H^{2s}}\|b\|_{H^{2s}}\|b\|_{H^{2s-1}}\\
  \lesssim &\| u\|_{H^{2s}} (1+t)^2\|b\|_{H^{2s}}^2 .
\end{split}
\end{equation}
Hence,
\begin{equation}\label{eqJ2}
\begin{split}
\int_{0}^{t} |N_2(\tau)| \; d\tau \lesssim &
\sup_{0 \leq \tau \leq t} \|u(\tau)\|_{H^{2s}} \int_{0}^{t} (1+\tau)^2\|b\|_{H^{2s}}^2\; d\tau \\
\lesssim &
\mathcal{E}_0^{1/2}(t)\mathcal{E}_1(t).
\end{split}
\end{equation}

Also for the term $N_3$, notice $s \geq 2$, we directly know that
\begin{equation}\nonumber
\begin{split}
  |N_3| \lesssim & (1+t)^2\|b\|_{W^{s,\infty}} \|b\|_{H^{2s}}\|u\|_{H^{2s-1}} \\
  \lesssim &(1+t)^2\|b\|_{H^{2s}}^2 \|u\|_{H^{2s-1}}.
\end{split}
\end{equation}
Hence,
\begin{equation}\label{eqJ3}
\begin{split}
\int_{0}^{t} |N_3(\tau)| \; d\tau \lesssim &
\sup_{0 \leq \tau \leq t} \|u(\tau)\|_{H^{2s-1}} \int_{0}^{t} (1+\tau)^2\|b\|_{H^{2s}}^2 \; d\tau\\
\lesssim &
\mathcal{E}_0^{1/2}(t)\mathcal{E}_1(t).
\end{split}
\end{equation}

Next, we turn to the last but wildest term $N_4$. We shall first rewrite it into two terms,
\begin{equation}\nonumber
\begin{split}
N_4  =& -(1+t)^2 \int_{\mathbb{T}^2} u\cdot \nabla\nabla^{2s-1}u \nabla^{2s-1}u dx \\
&-(1+t)^2 \sum_{k=1}^{2s-1}C_{2s-1}^{k}\int_{\mathbb{T}^2}\nabla^k u\cdot \nabla\nabla^{2s-1-k}u \nabla^{2s-1}u  dx .
\end{split}
\end{equation}
Using the same method as in $M_3$, by Gagliardo--Nirenberg interpolation inequality and Proposition \ref{prop}, we get
\begin{equation}\nonumber
\begin{split}
  |N_4| \lesssim & (1+t)^2 \|\nabla u\|_{L^\infty} \|u\|_{H^{2s-1}}^2 \\
  \lesssim & \|\partial_2 u\|_{H^{2s}}^{1/4} \|\partial_2 u\|_{H^{2s-2}}^{3/4}(1+t)^2 \|u\|_{H^{2s-1}}^2,
\end{split}
\end{equation}
provided that $s\geq 2$. Thus, we obtain
\begin{equation}\label{eqJ4}
\begin{split}
&\int_{0}^{t} |N_4 (\tau)| \; d\tau\\
 \lesssim& \sup_{0\leq \tau \leq t}  (1+\tau)^2 \|u(\tau)\|_{H^{2s-1}}^2\int_{0}^{t} \|\partial_2 u\|_{H^{2s}}^{1/4} \|\partial_2 u\|_{H^{2s-2}}^{3/4} \; d\tau \\
\lesssim& \mathcal{E}_0^{1/8}(t)\mathcal{E}_1^{11/8}(t).
\end{split}
\end{equation}

Now, summing up the estimates for $N_1 \thicksim N_4$, thus \eqref{eqJ1}, \eqref{eqJ2}, \eqref{eqJ3} and \eqref{eqJ4}. Integrating \eqref{lem2.2} with time, we get the estimate of $\mathcal{E}_{1,1}$
\begin{equation}\label{eqE111}
\mathcal{E}_{1,1}(t)\lesssim  \mathcal{E}_1(0)+\mathcal{E}_0^{1/2}(t)\mathcal{E}_1^{1/2}(t)+\mathcal{E}_0^{1/2}(t)\mathcal{E}_1(t)+\mathcal{E}_0^{1/8}(t)\mathcal{E}_1^{11/8}(t).
\end{equation}
Here, we have used the Poincar$\mathrm{\acute{e}}$ inequality to consider the highest order term only.

Totally like the process in the proof of Lemma 3.1, operating $\nabla^{2s-2}$ derivative on the first equation of \eqref{eq1.2} and taking
inner product with $\nabla^{2s-2} \partial_2 u$, adding the time weigh $(1+t)^2$, we get

\begin{equation}\label{eqlem2.2}
(1+t)^2 \|\partial_2 u\|_{\dot H^{2s-2}}^2 = N_5 + N_6 + N_7,
\end{equation}
where,
\begin{equation}\nonumber
\begin{split}
  N_5 =& (1+t)^2\int_{\mathbb{T}^2} \nabla^{2s-2} (u\cdot \nabla b-b\cdot \nabla u) \nabla^{2s-2}\partial_2 u \; dx, \\
  N_6 =& -(1+t)^2 \int_{\mathbb{T}^2}\nabla^{2s-2} \Delta b \nabla^{2s-2}\partial_2 u \; dx ,\\
  N_7 =& (1+t)^2 \int_{\mathbb{T}^2} \nabla^{2s-2} b_t \nabla^{2s-2}\partial_2 u \; dx .
\end{split}
\end{equation}

Using H\"{o}lder inequality and Sobolev imbedding theorem, we easily get, for $s \geq 2$
\begin{equation}\nonumber
\begin{split}
  |N_5| \lesssim & (1+t)^2\|u\|_{W^{s-1,\infty}} \|b\|_{H^{2s-1}}\|\partial_2 u\|_{H^{2s-2}} \\
  & +(1+t)^2\|b\|_{W^{s-1,\infty}} \|u\|_{H^{2s-1}}\|\partial_2 u\|_{H^{2s-2}}   \\
  \lesssim &(1+t)^2\|u\|_{H^{2s-1}}\|b\|_{H^{2s-1}}\|\partial_2 u\|_{H^{2s-2}}.
\end{split}
\end{equation}
Hence, we conclude
\begin{equation}\label{eqJ5}
\begin{split}
&\int_{0}^{t} |N_5(\tau)| \; d\tau\\
 \lesssim & \sup_{0 \leq \tau \leq t} \|u(\tau)\|_{H^{2s-1}} \int_{0}^{t} (1+\tau)^2 \|b\|_{H^{2s-1}}\|\partial_2 u\|_{H^{2s-2 }}\; d\tau \\
\lesssim & \mathcal{E}_0^{1/2}(t)\mathcal{E}_1(t) .
\end{split}
\end{equation}

Similarly, we have
\begin{equation}\label{eqJ6}
  \begin{split}
\int_{0}^{t} |N_6(\tau)| \; d\tau \lesssim & \int_{0}^{t} (1+\tau)^2\|b\|_{H^{2s}}\|\partial_2 u\|_{H^{2s-2}} \; d\tau\\
\lesssim & \mathcal{E}_{1,1}^{1/2} (t) \Big( \int_{0}^{t}\|\partial_2 u\|_{H^{2s-2}}^2 \; d\tau \Big)^{1/2} .
\end{split}
\end{equation}

For the last term $N_7$, we shall first rewrite it. Using integration by parts, we have
\begin{equation}\label{eqj7}
\begin{split}
  N_7 =& \frac{d}{dt} (1+t)^2\int_{\mathbb{T}^2} \nabla^{2s-2}b \nabla^{2s-2}\partial_2 u \;dx \\ &-2(1+t)\int_{\mathbb{T}^2} \nabla^{2s-2}b \nabla^{2s-2}\partial_2 u \;dx \\
  &+(1+t)^2 \int_{\mathbb{T}^2} \nabla^{2s-2}\partial_2 b \nabla^{2s-2 }u_t \;dx .
\end{split}
\end{equation}
 Notice the second equality of \eqref{eq1.2}, the last term on the right hand side of \eqref{eqj7} can be bounded
\begin{equation}\nonumber
\begin{split}
  & (1+t)^2 \int_{\mathbb{T}^2}  \nabla^{2s-2}\partial_2 b \nabla^{2s-2} \big( b\cdot \nabla b+\partial_2 b-u\cdot \nabla u-\nabla p\big) \; dx \\
  \lesssim & (1+t)^2\|b\|_{W^{s-1,\infty}}\|b\|_{H^{2s-1}}^2 +(1+t)^2 \|b\|_{H^{2s-1}}^2 \\
  &+(1+t)^2\|b\|_{H^{2s-1}} \|\nabla u\|_{W^{s-2,\infty}}\|u\|_{H^{2s-1}} \\
  \lesssim & (1+t)^2 \|b\|_{H^{2s-2}}\|b\|_{H^{2s-1}}^2 +(1+t)^2 \|b\|_{H^{2s-1}}^2\\
  &+(1+t)^2 \|b\|_{H^{2s-1}} \| \partial_2 u\|_{H^{2s}}\|u\|_{H^{2s-1}},
\end{split}
\end{equation}
provided that $ s\geq 2$. Thus, the estimate of $N_7$ can be concluded
\begin{equation}\label{eqJ7}
\begin{split}
|\int_{0}^{t} N_7(\tau) \; d\tau| \lesssim &\sup_{0 \leq \tau \leq t} (1+\tau)^2 \|b(\tau)\|_{H^{2s-2}}
 \|\partial_2 u\|_{H^{2s-2}}\\
 & + \int_{0}^{t} (1+\tau) \|b\|_{H^{2s-2}} \|\partial_2 u\|_{H^{2s-2}} \; d\tau \\
 & +\sup_{0 \leq \tau \leq t}\|b(\tau)\|_{H^{2s-2}} \int_{0}^{t} (1+\tau)^2 \|b\|_{H^{2s-1}}^2 \; d\tau \\
 & + \int_{0}^{t} (1+\tau)^2 \|b\|_{H^{2s-1}}^2 \; d\tau\\
 & + \sup_{0 \leq \tau \leq t} (1+\tau) \|u\|_{H^{2s-1}} \int_{0}^{t} (1+\tau) \|b\|_{H^{2s-1}} \|\partial_2 u\|_{H^{2s}} \; d\tau \\
 \lesssim &\mathcal{E}_{1,1}(t)+\mathcal{E}_0^{1/2}(t)\mathcal{E}_1^{1/2}(t)+\mathcal{E}_0^{1/2}(t)\mathcal{E}_1(t).
 \end{split}
\end{equation}

Integrating \eqref{eqlem2.2} with time, using \eqref{eqJ5}, \eqref{eqJ6}, \eqref{eqJ7} and Young inequality we get
\begin{equation}\label{eqE12}
\int_{0}^{t}  (1+\tau)^2 \|\partial_2 u\|_{H^{2s-2}}^2 \; d\tau \lesssim \mathcal{E}_{1,1}+\mathcal{E}_0^{1/2}(t)\mathcal{E}_1^{1/2}(t)+\mathcal{E}_0^{1/2}(t)\mathcal{E}_1(t).
\end{equation}

Now, multiplying \eqref{eqE111} by suitable large number and plus \eqref{eqE12}, using Young inequality, we complete the proof of this lemma.
\end{proof}

\subsection{Proof of the Theorem \ref{thm1}}

First, we define the total energy as follows
$$ \mathcal{E}(t)=\mathcal{E}_0(t) +  \mathcal{E}_1(t). $$
Multiplying \eqref{E_0} and \eqref{E_1} in the above lemmas by different suitable number, and summing them up, we can get
\begin{equation}\label{eqE}
\mathcal{E}(t)\leq  C_1\mathcal{E}(0)+C_1 \mathcal{E}^{3/2}(t),
\end{equation}
for some positive constant $C_1$.

Under the setting of initial data \eqref{setofdata}, there exists a positive constant $C_2$ such that the initial total energy $ \mathcal{E}(0) \leq C_2 \epsilon$. According to the standard local well-posedness theory which can be obtained by classical arguments, there exists a positive time $T$ such that for $C_3=C_1C_2$,
\begin{equation}\label{eqlocal}
 \mathcal{E}(t) \leq 2C_3 \epsilon, \qquad \forall \; t\in [0,T].
\end{equation}
Let $T^*$ be the largest possible time of $T$ satisfying \eqref{eqlocal}, it's then to show $T^*=\infty$. Notice the estimate \eqref{eqE}, we can use a standard continuation argument to get the desired result provided that $\epsilon$ is small enough. Hence, we finish the proof of Theorem 1.1.

\section*{Acknowledgement}

The first author is supported by Key Laboratory of Mathematics for Nonlinear Sciences (Fudan University), Ministry of Education of China, Shanghai, Key Laboratory for Contemporary Applied Mathematics, School of Mathematical Sciences, Fudan University, NSFC under grant No.11421061, 973 Program (grant No.2013CB834100) and 111 project.

\end{document}